\documentclass[a4paper,11pt]{amsart}
\usepackage{eucal}
\usepackage{amsmath,amsthm,amssymb}
\usepackage{amsfonts}
\usepackage{latexsym}
\usepackage{eucal}
\usepackage{mathrsfs}
\usepackage{txfonts}
\theoremstyle{plain}
\newtheorem{theorem}{Theorem}[section]
\newtheorem{proposition}[theorem]{Proposition}

\newtheorem{corollary}[theorem]{Corollary}

\theoremstyle{definition}
\newtheorem{definition}[theorem]{Definition}

\makeatletter
\renewenvironment{proof}[1][\proofname]{\par
  \normalfont
  \topsep6\p@\@plus6\p@ \trivlist
  \item[\hskip\labelsep{\bfseries #1}\@addpunct{\bfseries.}]\ignorespaces
}{%
  \endtrivlist
}
\renewcommand{\proofname}{proof}%\renewcommand{\qedsymbol}{\hfill$\blacksquare$}
\theoremstyle{remark}
\newtheorem{remark}[theorem]{Remark}
\numberwithin{equation}{section}

\title{Non-unital algebra objects of stable symmetric monoidal model categories by Smith ideal theory}
\date{\today} 
\author{Yuki Kato}
\address{National institute of technology, Ube college, 
	      2-14-1, Tokiwadai, Ube, Yamaguchi, JAPAN 755-8555.}
\email{ykato@ube-k.ac.jp}
\keywords{Symmetric monoidal model categories, Smith ideals, Non-unital algebras}

\newcommand{\Alg}{\mathrm{Alg}}

\newcommand{\Fun}{\mathrm{Fun}}

\usepackage{enumerate}
\usepackage{array}
\usepackage[all]{xy} %\xy-pic
\usepackage{delarray}
\def\qed{{\hfill $\Box$}}
\begin{document}
\thispagestyle{empty}
\begin{abstract}
This note remarks that the correspondence between  non-unital
algebras and augmented unital algebras can be derived from Hovey's Smith
ideal theory.  Applying Smith ideal theory of stable symmetric monoidal
model category, we formulate non-unital  algebra objects of
stable symmetric monoidal model categories and generalize the
correspondence between non-unital algebra objects and
augmented algebra objects.
\end{abstract}

\maketitle
%\tableofcontents
%#! platex Smith-nonunital.tex;dvipdfmx Smith-nonunital.dvi
\section{Introduction} 
\label{sec-Introduction} A ring without the assumption existence of a
multiplicative unit $1$ is called a {\it non-unital} ring. Any unital
ring is regarded as a non-unital ring by forgetting the existence of
$1$. This paper fixes a base unital ring $V$ and let $\Alg^{\rm
nu}_{V}$ denote the category of non-unital $V$-algebras.  An {\it
augmentation} $\varepsilon_A: A \to V$ is a ring homomorphism that the
unit homomorphism $V \to A$ is a section of $\varepsilon_A$, and a
$V$-algebra with an augmentation is called an {\it augmented
$V$-algebra}, denoting $\Alg_{V/ / V}$ the category of augmented
$V$-algebras. It is well-known that there exists a categorical
equivalence between $\Alg^{\rm nu}_V$ and $\Alg_{V/ / V}$: For any  $V$-algebra $A$, the direct sum $V \oplus A$ has a canonical
unital ring structure defined by
\[
 (m,\,a) \cdot (n,\,b) = (mn,\, na+mb+ab)  
\]
for $m, \, n \in V$ and $a,\,b \in A$.  It is easily
checked that the functor the adjunction
\[
  V \oplus (-) :  \Alg^{\rm nu}_{V} \rightleftarrows  \Alg_{V/ / V}: \mathrm{Ker} (\epsilon_{(-)} : - \to V)
\] 
is a pair of categorical equivalences, where $\mathrm{Ker} (\varepsilon_{(-)}: - \to V)$ is the kernel functor of augmentations. 
%Lurie~\cite[Section 5.4.4]{HA} generalized the definition of non-unital
%algebra objects for stable monoidal $\infty$-category over an
%$\infty$-operads. 

The main result of this paper is a model categorical analogue of the
above categorical equivalence (Theorem~\ref{mainThm}): A {\it stable}
model category is a model category whose homotopy category is
triangulated. We generalize the definition of non-unital 
algebras for stable symmetric monoidal model categories by using Hovey's
Smith ideal theory~\cite[Theorem 4.3]{Smith-ideals}.  An advantage of
Smith ideal theory of stable symmetric monoidal model categories is that
it enables us to clarify that a symmetric monoidal model structure of
non-unital algebra objects is the push-out products induced by ones.

\subsubsection*{Acknowledgements}
The author was supported by Grants-in-Aid for Scientific Research No.23K03080, Japan Society for the Promotion of Science.

\section{Smith ideal theory of symmetric monoidal model categories}
\label{sec:Smith} 

We explain Hovey~\cite{Smith-ideals}'s Smith ideal theory of symmetric
monoidal categories. This section assumes that categories are always
pointed, and $0$ denotes the zero object.
\subsection{Smith ideal theory of pointed categories}
For any symmetric monoidal category $\mathcal{C}$, let $\Alg(\mathcal{C})$
denote the subcategory of $\mathcal{C}$ spanned by  monoid
objects in the sense of MacLane~\cite{zbMATH00195199}.

Let $[1]$ denote the category with two objects $0$ and $1$,
and only one non-identity morphism $0 \to 1$. For any pointed category
$\mathcal{C}$, the diagram category $\Fun([1],\,\mathcal{C})$ is called
the arrow category, and $\mathrm{Ar}(\mathcal{C})$ denotes it.  The symmetric monoidal structure of $\mathcal{C}$ inherits $\mathrm{Ar}(\mathcal{C})$ two symmetric monoidal structures: For any two morphisms in
$\mathcal{C}$, $f:X_0 \to X_1$ and $g: Y_0 \to Y_1$, the {\it tensor
product monoidal structure} is defined by $f \otimes g$ as the
composition
\[
f \otimes g : X_0 \otimes Y_0 \to X_1 \otimes Y_1,
\] 
and another is the {\it push-out product monoidal structure} defined by
the induced morphism:
\[
 f \Box g: ( X_0 \otimes Y_1   ) \amalg_{ X_0 \otimes Y_0 } (   X_1 \otimes Y_0 ) \to  X_1 \otimes Y_1. 
\]
We let $\mathrm{Ar}^\otimes (\mathcal{C})$ the arrow category with the tensor product monoidal structure and $\mathrm{Ar}^\Box (\mathcal{C})$ with the push-out monoidal structure. 
Hovey proved the cokernel functor is strongly monoidal and admits a lax monoidal right adjoint:
\begin{theorem}[\cite{Smith-ideals} Theorem 1.4]
\label{HoveySmith} Let $\mathcal{C}$ be a pointed closed symmetric
monoidal category. Then the functor
\[
\mathrm{cok} : \mathrm{Ar}^\Box(\mathcal{C}) \ni (f:X \to Y)   \mapsto (\mathrm{cok}(f ): Y \to \mathrm{Coker}(f)) \in 
\mathrm{Ar}^\otimes(\mathcal{C})
\]
is strongly symmetric monoidal, and its right adjoint is the kernel functor $\mathrm{ker}: (f: X \to Y) \mapsto (\mathrm{ker}(f): \mathrm{Ker}(f) \to X)$.
 \qed 
\end{theorem}
Hovey mentioned in \cite{Smith-ideals}, in general, the kernel functor $\mathrm{ker}: \mathrm{Ar}(\mathcal{C}) \to \mathrm{Ar}(\mathcal{C})$ is lax symmetric monoidal: A canonical functor 
\[
  \mathrm{ker}(f) \Box \mathrm{Ker}(g) \to \mathrm{ker}(f \otimes g)
\]
is induced by the canonical morphism $\mathrm{cok}( \mathrm{ker}(f) \Box \mathrm{Ker}(g)) \simeq \mathrm{cok}( \mathrm{ker}(f) )\otimes \mathrm{cok}( \mathrm{ker}(g) ) \to f \otimes g$ for any $f,\, g \in \mathrm{Ar}^\otimes(\mathcal{C})$. 
Hence those  functors  $\mathrm{cok}$ and $\mathrm{ker}$ send monoid objects to monoid objects. A {\it Smith ideal} is a monoid object of the symmetric monoidal category $\mathrm{Ar}(\mathcal{C})$ with respect to the push-out product monoidal
structure. 

\subsection{Definition of  non-unital algebra objects of symmetric monoidal categories}

For any pointed closed symmetric monoidal category $\mathcal{C}$ with a monoidal unit $V$, let $\Alg(\mathcal{C})_{V//V}$ denote the full subcategory of
$\mathrm{Ar}^\otimes(\mathcal{C})$ spanned by augmentations of algebra objects of the tensor product monoidal structure.

Let $\mathcal{C}$ be a pointed closed symmetric monoidal category and
$\mathrm{Ar}^{\rm im}(\mathcal{C})$ denotes the full subcategory spanned
those objects $f:R \to S$ such that the unit $f \to \mathrm{ker} (
\mathrm{cok}(f))$ is an isomorphism. Then the cokernel functor
$\mathrm{Ar}^{\rm im}(\mathcal{C}) \to \mathrm{Ar}^{\rm
coim}(\mathcal{C})$ is a categorical equivalence.

\begin{proposition}
\label{stabilization} Let $\mathcal{C}$ be a locally presentable abelian
symmetric monoidal category. Then the arrow category
$\mathrm{Ar}(\mathcal{C})$ is also locally presentable and there exist
reflective localization functors $L_{\rm im}:
\mathrm{Ar}^\Box(\mathcal{C}) \to \mathrm{Ar}^{\rm im}(\mathcal{C}) $ by
the unit $\mathrm{id} \to \mathrm{im}$ and $L_{\rm coim}:
\mathrm{Ar}^\otimes(\mathcal{M}) \to \mathrm{Ar}^{\rm coim}(\mathcal{M})
$ by the counit $\mathrm{coim} \to \mathrm{id}$ such that the adjunction
\[ 
\mathrm{cok} :
\mathrm{Ar}^\Box(\mathcal{C}) \rightleftarrows \mathrm{Ar}^\otimes(\mathcal{C}):
\mathrm{ker}
\] 
induces categorical equivalences
 \[
  \mathrm{cok} : \mathrm{Ar}^{\rm im}(\mathcal{C})
 \rightleftarrows \mathrm{Ar}^{\rm coim}(\mathcal{C}): \mathrm{ker}.
\]
\end{proposition}
\begin{proof}
Since any abelian category is binormal, $f \to \mathrm{im}(f)$ is an
isomorphism if and only if $f$ is a monomorphism.  By
Ad{\'a}mek--Rosick{\'y}~\cite[p.44, Corollary 1.5.4]{LocPresentable},
the arrow category is locally presentable, and it admits reflective
localization.  By the definitions of those functors: $\mathrm{im}=
\mathrm{ker} \circ \mathrm{cok}$ and $\mathrm{coim}= \mathrm{cok} \circ
\mathrm{ker}$, the restriction $ \mathrm{cok} : \mathrm{Ar}^{\rm
im}(\mathcal{C}) \to \mathrm{Ar}^{\rm coim}(\mathcal{C}) $ are
quasi-inverse functors. \qed
\end{proof}

\begin{definition}
 Let $\mathcal{C}$ be a locally presentable symmetric monoidal abelian
 category with a monoidal unit $V$. A non-unital monoid
 object of $\mathcal{C}$ is a Smith ideal $j:I \to R$ in the localized
 full subcategory $\mathrm{Ar}^{\rm im}(\mathcal{C})$ satisfying that
 the cokernel $\mathrm{cok}(j): R \to \mathrm{Coker}(j) $ is isomorphic
 with an argumentation $\varepsilon_R: R \to V$, and $\Alg^{\rm
 nu}(\mathcal{C})$ denote the full subcategory of
 $\Alg(\mathrm{Ar}^\Box(\mathcal{C}))$ spanned by non-unital
  algebra objects, where $\mathrm{Ar}^\Box(\mathcal{C})$
 denotes the arrow category of $\mathcal{C}$ whose monoidal structure is
 the push-out monoidal structure.
\end{definition}

\begin{proposition}
\label{lem-Im}
Further, let $\Alg^{\rm nu}(\mathcal{C})$ denote the full subcategory of
$\mathrm{Ar}^{\Box}(\mathcal{C})$ the full subcategory which is the
essential image of the restriction of the kernel functor on
$\Alg(\mathcal{C})_{V//V}$. Then any object $j:I \to R$ of $\Alg^{\rm nu}(\mathcal{C})$, the unit $j \to \mathrm{im}(j)$ is an isomorphism. 
\end{proposition}
\begin{proof}
By definition of the category $\Alg^{\rm nu}(\mathcal{C})$, one has an isomorphism $V \oplus I \to R$. The assertion is clear. \qed
\end{proof}
\begin{corollary}
Let $\mathcal{C}$ be a locally presentable symmetric monoidal abelian model category. The adjunction 
\[ 
\mathrm{cok} :
\mathrm{Ar}^\Box(\mathcal{C}) \rightleftarrows \mathrm{Ar}^\otimes(\mathcal{C}):
\mathrm{ker}
\] 
induces categorical equivalences 
  \[
 \mathrm{cok}:  \Alg^{\rm
 nu}(\mathcal{C}) \to  \Alg(\mathcal{C})_{V//V} : \mathrm{ker}.
\] \qed
\end{corollary}

\section{Non-unital commutative algebra objects of symmetric monoidal  model categories} 
A symmetric monoidal model category $\mathcal{M}$ is a model category
with a symmetric monoidal structure $ -\otimes - : \mathcal{M} \times
\mathcal{M} \to \mathcal{M} $ such that, for any object $M$ of
$\mathcal{M}$, those functors $(-) \otimes M$ and $M \otimes (-)$ are
left Quillen functors on $\mathcal{M}$. 

\subsection{The arrow categories of pointed symmetric monoidal model categories}

The category $\mathrm{Ar}(\mathcal{M})$ has two canonical model
structures, the {\it injective model structure} and the {\it projective
model structure} induced by $\mathcal{M}$'s:
\begin{definition}
Let $\mathcal{M}$ be a model category. The arrow category
$\mathrm{Ar}(\mathcal{M})$ has the following two model structures.
\begin{itemize}
 \item (Injective model structure) A morphism $\alpha:(f: X_0 \to X_1)
       \to (g: Y_0 \to Y_1) $ is a cofibrations (resp. weak equivalence)
       in $\mathrm{Ar}(\mathcal{M}) $ if and only if so is each
       $\mathrm{Ev}_i(\alpha)$ for $i= 0,\,1$.  Fibrations are morphisms
       with the right lifting property for all trivial cofibrations.
 \item (Projective model structure) A morphism $\alpha:(f: X_0 \to X_1)
       \to (g: Y_0 \to Y_1) $ is a fibrations (resp. weak equivalence)
       in $\mathrm{Ar}(\mathcal{M}) $ if and only if so is each
       $\mathrm{Ev}_i(\alpha)$ for $i= 0,\,1$.  Cofibrations are
       morphisms with the right lifting property for all trivial
       fibrations.
\end{itemize}
\end{definition}

In a pointed model category $\mathcal{M}$, we consider a
homotopically commutative diagram:
\[
 \xymatrix@1{ 
X \ar[r]^f \ar[d] & Y \ar[d]^g \\
0  \ar[r]  & Z.
}
\]
If the diagram is a homotopy Cartesian square, then $X$ is said to be a
homotopy {\it kernel} of $g$, and if it is a homotopy coCartesian
square, then $Z$ is a homotopy {\it cokernel} of $f$.  We can consider
homotopy image objects and homotopy coimage objects as additive
categories.

On the arrow category $\mathrm{Ar}(\mathcal{M})$, those functors
$\mathrm{cok}: \mathrm{Ar}(\mathcal{M}) \to \mathrm{Ar}(\mathcal{M})$
and $\mathrm{ker}: \mathrm{Ar}^\otimes(\mathcal{M}) \to
\mathrm{Ar}^\Box(\mathcal{M})$ are defined as follows: For a morphism $f : X
\to Y$ in $\mathcal{M}$, the arrow $\mathrm{cok}(f)$ is $Y \to
\mathrm{Coker}f$ and $\mathrm{ker}(f)$ $\mathrm{Ker}(f) \to X$. Then the
pair 
\[
 \mathrm{cok} : \mathrm{Ar}(\mathcal{M}) \rightleftarrows
\mathrm{Ar}(\mathcal{M}) : \mathrm{ker} 
\]
is a Quillen adjunction. 

\begin{definition}
Let $\mathcal{M}$ be a pointed symmetric monoidal model category. A {\it
Smith ideal} in $\mathcal{M}$ is a monoid object $j:I \to R$ in the
symmetric monoidal model category $\mathrm{Ar}^\Box(\mathcal{M})$ with
respect to the push-out product monoidal model structure.
\end{definition}

 We say that a Smith ideal $j:I \to R$ is {\it unit cokernel}
if the cokernel of $j$ is isomorphic to the monoidal unit object $V$ and
$\mathrm{cok}(j): R \to \mathrm{Coker}(j)$ is an augmentation of the
unit morphism $V \to R$.  Let $\Alg^{\rm nu}(\mathcal{C})$
denote the full subcategory of $\Alg(\mathrm{Ar}^\Box(\mathcal{C}))$ spanned
by unit cokernel Smith ideals. 
\begin{definition}
Let $\mathcal{M}$ be a stable symmetric monoidal model category with a
monoidal unit $V$ and $\Alg^{\rm nu}(\mathcal{M})$ denote the full
subcategory $\Alg(\mathrm{Ar}^\Box(\mathcal{M}))$ spanned by Smith
ideals whose cokernels are weakly equivalent to $V$. We say that an
object of $\Alg^{\rm nu}(\mathcal{M})$ is a {\it non-unital commutative
algebra object} of $\mathcal{M}$.
\end{definition}

\begin{theorem}
\label{mainThm}
Let $\mathcal{M}$ be a stable symmetric monoidal model category with a monoidal unit $V$. Then the Quillen equivalence
\[
  \mathrm{cok}:\mathrm{Ar}(\mathcal{M}) \rightleftarrows \mathrm{Ar}(\mathcal{M}) : \mathrm{ker}
\] 
induces a left Quillen equivalence between $\mathrm{cok}:\Alg^{\rm nu}(\mathcal{M}) \to \Alg(\mathcal{M})_{V//V}$. 
\end{theorem}
\begin{proof}
By using the second statement of \cite[Theorem 4.3]{Smith-ideals}, the cokernel
functor $ \mathrm{cok}:\mathrm{Ar}(\mathcal{M}) \to
\mathrm{Ar}(\mathcal{M})$ induces the left Quillen functor
$\mathrm{cok}:\Alg^{\rm nu}(\mathcal{M}) \to
\Alg(\mathcal{M})_{V//V}$, being essentially surjective on the homotopy
categories by definition. Since the unit
$u:\mathrm{Id}_{\mathrm{Ar}(\mathcal{M})} \to \mathrm{ker} \circ
\mathrm{cok}$ is a weak equivalence by \cite[Theorem 4.3]{Smith-ideals} again, the induced functor 
$\mathrm{cok}:\Alg^{\rm nu}(\mathcal{M}) \to \Alg(\mathcal{M})_{V//V}
$ is homotopically fully faithful. \qed
\end{proof}

\begin{remark}
If $\mathcal{M}$ is a locally presentable stable symmetric monoidal model category, the presentable $\infty$-category of augmented algebra is represented by the model category $\Alg(\mathcal{M})_{V//V}$. Hence the left-hand-side $\Alg^{\rm nu}(\mathcal{M})$ is equivalent to the $\infty$-category defined by Lurie~\cite[p.949, Definition 5.4.4.9 and Proposition 5.4.4.10]{HA}.       
\end{remark}

\bibliographystyle{alphadin}
%\bibliographystyle{plain}
%\bibliography{bibkato}

\nocite{Hirschhorn}
\nocite{Hoveybook}
%\nocite{HT}
%\nocite{HA}
\end{document}